\documentclass{article}
\usepackage[utf8]{inputenc}

\usepackage{amsmath, amssymb, amsthm}
\usepackage{mathtools,lineno}
\usepackage[dvipsnames]{xcolor}
\usepackage{graphicx}
\usepackage{tikz}
\usetikzlibrary{calc}

\newtheorem{theorem}{Theorem}
\newtheorem{proposition}{Proposition}
\newtheorem{lemma}{Lemma}

\theoremstyle{definition}

\newtheorem{conjecture}{Conjecture}

\newcommand{\rb}{\text{rb}}
\renewcommand{\mod}[1]{\;(\text{mod }#1)}

\title{Bounds for Rainbow-uncommon Graphs}
\author{Blake Bates  \and  Zhanar Berikkyzy \and  Nick Chiem \and Gabriel Elvin \and Risa Fines \and Maja Lie \and Hanna Mikulás \and Isaac Reiter \and Kevin Zhou}
\date{March 6, 2024}

\begin{document}

\maketitle

\begin{abstract}
We say a graph $H$ is $r$-rainbow-uncommon if the maximum number of 
rainbow copies of $H$ under an $r$-coloring of $E(K_n)$ is asymptotically 
(as $n \to \infty$)
greater than what is expected from uniformly random $r$-colorings.
Via explicit constructions, we show that for $H\in\{K_3,K_4, K_5\}$, $H$ is $r$-rainbow-uncommon for all $r\geq {|V(H)|\choose 2}$.
We also construct colorings to show that for $t \geq 6$, $K_t$ is $r$-rainbow-uncommon for sufficiently large $r$. 
\end{abstract}

\section{Background}

A classic starting point in Ramsey theory is the following question:
if any two people can identify as either strangers or acquaintances,
what is the minimum number of people needed to guarantee that there
exists a set of three people who are all either complete strangers or
mutual acquaintances? Reframed in the language of graph theory,
we can let vertices denote people and draw a blue edge between 
two people if they are strangers and a red edge if they are 
acquaintances. Then the question is equivalent to: what is the smallest
$n$ such that every coloring of the edges of a complete graph $K_n$ contains a 
monochromatic triangle, i.e. a copy of $K_3$ whose edges are either all
blue or all red? The answer to this question is $n = 6$, and in 
1959, Goodman asked and answered a more general follow-up question:
for an arbitrary $n$, 
what is the minimum number of monochromatic triangles over all  
2-colorings $E(K_n)$ \cite{Goodman}. He found an explicit formula, but a consequence of his result is that the minimum \textit{proportion}
of monochromatic triangles, that is
\[\frac{\text{number of monochromatic $K_3$}}{\text{total number of $K_3$}},\]
approaches $1/4$ as $n \to \infty$.
This is perhaps surprising, since $1/4$ is the expected proportion
in a uniformly random coloring, i.e. the minimum can be achieved,
at least asymptotically, by coloring completely randomly. 
This led to a new vein of Ramsey theory, classifying graphs as either
\textit{common} or \textit{uncommon}: common if the asymptotic
minimum can be achieved by uniformly random colorings, and uncommon
if better colorings can be found. See \cite[Section 2.6]{Conlon} for further reading in this area.

This question can be ``flipped'': given a fixed graph $H$ and $r$ colors,
what is the best way to color the edges of $K_n$ ($n$ large) in 
order to \textit{maximize} the number of \textit{rainbow} copies of
$H$ in $K_n$? By \textit{rainbow}, we mean that each edge of $H$ is 
assigned a different color from all the others. 
Some progress has been made in this area, 
and our work builds on the work in \cite{DeSilva}. 
Next, we paraphrase those authors' definitions, slightly
altered to suit our needs.

In this paper all graphs are simple. If $H$ and $G$ are graphs, 
a \textbf{copy} of $H$ in $G$ is a subgraph of $G$ that is isomorphic to $H$.
We define an \textbf{$r$-coloring} of a set $X$ to be a surjective
function $c : X \to \{1,\dots,r\}$.
For any graph $G$, let $E(G)$ denote the set of edges of $G$. 
Given a graph $G$ and subgraph $H$, we say $H$ is \textbf{rainbow}
if, under a coloring $c : E(G) \to \{1,\dots,r\}$, 
$c(uv) \neq c(u'v')$ for all distinct
$uv, u'v' \in E(H)$.

Results in this area are asymptotic in nature, so we clarify 
some notation used throughout. Let $f, g$ be functions of $n$.
If there exist constants $C, N$ such that $|f(n)| \leq Cg(n)$
for all $n \geq N$, we write $f = O(g)$.

For the remainder of this section, let $H$ be a fixed graph on
$t$ vertices with $e$ edges. We denote by $M_\rb(H; n, r)$ the maximum number of rainbow
copies of $H$ over all $r$-colorings of $E(K_n)$.
Furthermore, we define
\begin{align}
m_\rb(H; n, r) &\coloneqq \frac{M_\rb(H;n, r)}
{\text{total number of copies of $H$ in $K_n$}}, \\
m_\rb(H; r) &\coloneqq \lim_{n\to\infty} m_\rb(H; n, r). \label{def-rb-limit}
\end{align}
For a given graph $H$, $m_\rb(H; r)$ is known as the 
\textbf{$r$-rainbow Ramsey multiplicity constant}.
Note that it always exists due to the fact
that the sequence $m_\rb(H; n, r)$ is bounded and monotone in $n$, 
as shown in \cite{DeSilva}.
The overarching classification question is: for which graphs $H$
and number of colors $r$ can $m_{\rb}(H; r)$ be achieved by 
uniformly random colorings?
Note that, for any fixed copy of $H$ in a uniformly random $r$-coloring
of $E(K_n)$, there are $r\choose{e}$ distinct colors to choose that could
make the copy of $H$ rainbow, and once those colors are chosen there are
$e!$ ways the subgraph could be colored. Since there are $r^e$ possible colorings
of this subgraph, the expected proportion of rainbow $H$ in $G$ is ${r\choose{e}}e!/r^e$.
This motivates the following definition:
a graph $H$ is \textbf{$r$-rainbow-common} if
\begin{equation}\label{eq-common}
m_{\rb}(H; r) = \frac{{r\choose{e}}e!}{r^e}.
\end{equation}
Otherwise, $H$ is called \textbf{$r$-rainbow-uncommon}.

In this paper, we investigate rainbow-uncommonness of complete graphs,
building upon the following theorem.
\begin{theorem}[\cite{Erdos-Hajnal, DeSilva}]\label{thm-old}
$K_t$ is ${t\choose{2}}$-rainbow-uncommon for all $t \geq 3$.
\end{theorem}
The case $t = 3$ was proven by Erd\H{o}s and Hajnal in 1972 \cite{Erdos-Hajnal},
and De Silva et al. showed it for all $t\geq 4$ in 2019 \cite{DeSilva}.
Our first result extends that of Erd\H{o}s and Hajnal and the first
two cases from De Silva et al.
\begin{theorem}\label{thm-exact}
$K_t$ is $r$-rainbow-uncommon for all $r \geq {t\choose{2}}$,
when $t = 3, 4, 5$.
\end{theorem}
We also show that for $t \geq 6$, $K_t$ is $r$-rainbow-uncommon for all sufficiently
large $r$, stated precisely below.
\begin{theorem}\label{thm-weak}
Let $t \geq 6$. There exists $r_t$ such that
$K_t$ is $r$-rainbow-uncommon for all $r \geq r_t$.
\end{theorem}
In the next section, we prove Theorem \ref{thm-weak}.
Theorem \ref{thm-exact} is an immediate consequence of the work done 
in the proof of Theorem \ref{thm-weak},
which we show at the end of the section.

\section{Main Result}

In general, to show $r$-rainbow-uncommonness of a graph $H$, we can find an 
$r$-coloring of $E(K_n)$ with asymptotically strictly more rainbow
copies of $H$ than the number we would expect from a uniformly 
random coloring:
\begin{equation}\label{eq-random}
\frac{{n\choose{t}}t!}{|\text{Aut($H$)}|}\cdot\frac{{r\choose{e}}e!}{r^e}
= \frac{{r\choose{e}}e!}{|\text{Aut}(H)|r^e}n^t + O(n^{t-1})
\end{equation}
(the number of $H$ in $K_n$ times the probability that any given
$H$ is rainbow).

In our proofs, we use this fact and find explicit colorings
building on the iterated blow-up method in De Silva et al. \cite{DeSilva}.
The main idea is to find a coloring of $K_b$ which contains many
rainbow copies of $H$. Then we ``blow-up'' the $K_b$ by expanding
each vertex into a copy of $K_b$, and between each copy make all edges
the same color as the original edge between the original  two vertices.
We repeat this $k-1$ times to get a coloring of $K_n$, where $n = b^k$.
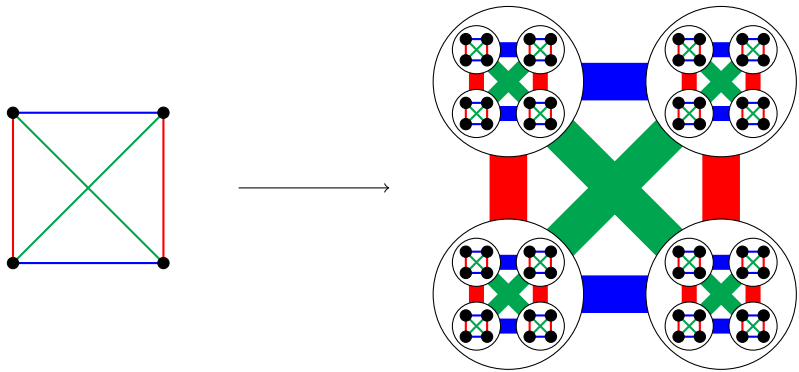
\begin{figure}
    \centering
\begin{tikzpicture}

\draw[thick, red] (-8,1) -- (-8,-1);
\draw[thick, red] (-6,1) -- (-6,-1);
\draw[thick, blue] (-8,1) -- (-6,1);
\draw[thick, blue] (-8,-1) -- (-6,-1);
\draw[thick, Green] (-8,1) -- (-6,-1);
\draw[thick, Green] (-8,-1) -- (-6,1);
\draw[fill=black] (-8,1) circle (0.075);
\draw[fill=black] (-8,-1) circle (0.075);
\draw[fill=black] (-6,1) circle (0.075);
\draw[fill=black] (-6,-1) circle (0.075);

\draw[->] (-5,0) -- (-3,0);

\draw[thick, red] (45: 2) -- (-45: 2);
\draw[thick, red] (135: 2) -- (-135: 2);
\draw[thick, blue] (45: 2) -- (135: 2);
\draw[thick, blue] (-45: 2) -- (-135: 2);
\draw[thick, Green] (135: 2) -- (-45: 2);
\draw[thick, Green] (45: 2) -- (-135: 2);
 
\draw[line width=5mm, red] (45: 2) -- (-45: 2);
\draw[line width=5mm, red] (135: 2) -- (-135: 2);
\draw[line width=5mm, blue] (45: 2) -- (135: 2);
\draw[line width=5mm, blue] (-45: 2) -- (-135: 2);
\draw[line width=5mm, Green] (135: 2) -- (-45: 2);
\draw[line width=5mm, Green] (45: 2) -- (-135: 2);
 
\foreach \t in {0,...,3}{
        \draw[fill=white] (90*\t + 45: 2) circle (1);
    
      
        \draw[thick, red] ($(90*\t + 45: 2) + (45: 0.6)$) -- ($(90*\t + 45: 2) + (-45: 0.6)$);
        \draw[thick, red] ($(90*\t + 45: 2) + (135: 0.6)$) -- ($(90*\t + 45: 2) + (-135: 0.6)$);
        \draw[thick, blue] ($(90*\t + 45: 2) + (45: 0.6)$) -- ($(90*\t + 45: 2) + (135: 0.6)$);
        \draw[thick, blue] ($(90*\t + 45: 2) + (-45: 0.6)$) -- ($(90*\t + 45: 2) + (-135: 0.6)$);
        \draw[thick, Green] ($(90*\t + 45: 2) + (135: 0.6)$) -- ($(90*\t + 45: 2) + (-45: 0.6)$);
        \draw[thick, Green] ($(90*\t + 45: 2) + (45: 0.6)$) -- ($(90*\t + 45: 2) + (-135: 0.6)$);
        
        \draw[line width=2mm, red] ($(90*\t + 45: 2) + (45: 0.6)$) -- ($(90*\t + 45: 2) + (-45: 0.6)$);
        \draw[line width=2mm, red] ($(90*\t + 45: 2) + (135: 0.6)$) -- ($(90*\t + 45: 2) + (-135: 0.6)$);
        \draw[line width=2mm, blue] ($(90*\t + 45: 2) + (45: 0.6)$) -- ($(90*\t + 45: 2) + (135: 0.6)$);
        \draw[line width=2mm, blue] ($(90*\t + 45: 2) + (-45: 0.6)$) -- ($(90*\t + 45: 2) + (-135: 0.6)$);
        \draw[line width=2mm, Green] ($(90*\t + 45: 2) + (135: 0.6)$) -- ($(90*\t + 45: 2) + (-45: 0.6)$);
        \draw[line width=2mm, Green] ($(90*\t + 45: 2) + (45: 0.6)$) -- ($(90*\t + 45: 2) + (-135: 0.6)$);
        
        \foreach \s in {0,...,3}{
               
               \draw[fill=white] (90*\t + 45: 2) + (90*\s + 45: 0.6) circle (0.32);


               \draw[thick, red] ($(90*\t + 45: 2) + (90*\s + 45: 0.6) + (45: 0.2)$) -- ($(90*\t + 45: 2) + (90*\s + 45: 0.6) + (-45: 0.2)$);
               \draw[thick, red] ($(90*\t + 45: 2) + (90*\s + 45: 0.6) + (135: 0.2)$) -- ($(90*\t + 45: 2) + (90*\s + 45: 0.6) + (-135: 0.2)$);
               \draw[thick, blue] ($(90*\t + 45: 2) + (90*\s + 45: 0.6) + (45: 0.2)$) -- ($(90*\t + 45: 2) + (90*\s + 45: 0.6) + (135: 0.2)$);
               \draw[thick, blue] ($(90*\t + 45: 2) + (90*\s + 45: 0.6) + (-45: 0.2)$) -- ($(90*\t + 45: 2) + (90*\s + 45: 0.6) + (-135: 0.2)$);
               \draw[thick, Green] ($(90*\t + 45: 2) + (90*\s + 45: 0.6) + (45: 0.2)$) -- ($(90*\t + 45: 2) + (90*\s + 45: 0.6) + (-135: 0.2)$);
               \draw[thick, Green] ($(90*\t + 45: 2) + (90*\s + 45: 0.6) + (-45: 0.2)$) -- ($(90*\t + 45: 2) + (90*\s + 45: 0.6) + (135: 0.2)$);
               
               \foreach \q in {0,...,3}{
                       
                       \draw[fill=black] (90*\t + 45: 2) ++ (90*\s + 45: 0.6) + (90*\q + 45: 0.2) circle (0.075);
                       
               }
        }
}
\end{tikzpicture}
    \caption{Iterated blowup of a $K_4$. The ``thick'' edges indicate that
    all possible edges between the sets of vertices are given the specified color.}
    \label{fig:iterated-blowup}
\end{figure}
An example of an iterated blowup of a $K_4$ is given in Figure \ref{fig:iterated-blowup}.
Note that this generates a coloring for only a subsequence of $K_n$s. However, since 
$m_\rb(H; n, r)$ is bounded and monotone in $n$ \cite{DeSilva}, the limit $m_\rb(H; r)$ must exist
and every subsequence must converge to that limit. 

The inequalities in the following lemma were used in \cite{DeSilva}, and we include their proofs here for completeness.

\begin{lemma}\label{lem-main}
Let $t \leq b$, let $a$ denote the number of rainbow
copies of $H$ in an $r$-coloring of $E(K_b)$, and let $F(n)$ denote the number of rainbow
copies of $H$ in the $K_n$ generated from an iterated blow-up of $K_b$. Then 
\begin{equation}\label{eq-blowup}
F(n) \geq bF(n/b) + a(n/b)^t.
\end{equation}
Furthermore, the solution to this recurrence is  
\begin{equation}\label{eq-solved-recurrence}
F(n)\geq \frac{a}{b^t -b}n^{t}+O(n^{t-1}).
\end{equation}
\end{lemma}

\begin{proof}
First, we derive the inequality (\ref{eq-blowup}).
Each vertex of $K_b$ that we blow up contains $F(n/b)$ rainbow copies of $H$,
contributing the $bF(n/b)$ term. Additionally, one vertex from each of the $b$
parts can be chosen in $(n/b)^t$ ways, and each choice contains $a$ rainbow copies
of $H$. Therefore, $F(n) \geq bF(n/b) + a(n/b)^t$.
For an example of this bound, see Figure \ref{fig:iterated-blowup}.
We solve this recurrence using generating functions. First, we transform the recurrence into something that is easier to solve by making the substitution $n=b^k$, where $k$ is the number of iterations in the blow-up. We define the function $f(k)=F(b^k)$ to simplify (\ref{eq-blowup}) and get the following recurrence: 

$$f(k) \geq bf\left(k-1\right) + a\left(b^{k-1}\right)^t.$$

Let $A(x)=\sum_{k=0}^{\infty} f(k)x^k$ be the generating function. Since $f(0)=0$, we have:

\begin{align*}
    A(x)&=\sum_{k=1}^{\infty} f(k)x^k \\
       &\geq \sum_{k=1}^{\infty} \left[bf(k-1)+a\left(b^{k-1}\right)^t\right]x^k \\
       &= b\sum_{k=1}^\infty f(k-1)x^k + a\sum_{k=1}^{\infty} (b^{k-1})^tx^k \\
       &= bx\sum_{k=1}^{\infty} f(k-1)x^{k-1} + \frac{a}{b^t}\sum_{k=1}^{\infty} (b^t x)^k \\
       &= bxA(x) + \frac{a}{b^t}\left[\sum_{k=0}^{\infty} (b^tx)^k -1\right] \\
       &= bxA(x) + \frac{a}{b^t}\left[\frac{1}{1-b^tx}-1\right] \\
       &= bxA(x) + \frac{ax}{1-b^tx}.
\end{align*}

Note that the geometric series $\sum (b^t x)^k$ converges for $|x| < b^{-t}$, which also ensures
$1 - bx > 0$. Therefore, isolating $A(x)$ we get
\begin{equation*}
A(x) \geq \frac{ax}{(1-bx)(1-b^tx)}.
\end{equation*}
Next, using partial fraction decomposition, we can rewrite $A(x)$ as a power series and obtain the closed form of the recurrence. For real numbers $B$ and $C$, if we set 
$$\frac{ax}{(1-bx)(1-b^t x)} = \frac{B}{1-bx} + \frac{C}{1-b^t x},$$ 
we obtain $B=\frac{a}{b-b^t}$ and $C=-\frac{a}{b-b^t}$. Therefore:

\begin{align*}
A(x)& \geq \frac{a}{b-b^t}\cdot \frac{1}{1-bx} - \frac{a}{b-b^t}\cdot \frac{1}{1-b^tx} \\ 
&= \frac{a}{b-b^t}\sum_{k=0}^{\infty} (bx)^k - \frac{a}{b-b^t}\sum_{k=0}^{\infty} (b^tx)^k \\ 
&= \sum_{k=0}^{\infty} \left[\frac{a\left(b^{tk}-b^k\right)}{b^t-b}\right]x^k. 
\end{align*} \\

Therefore, $f(k) \geq \frac{a\left(b^{tk}-b^k\right)}{b^t-b}$. Substituting $n = b^k$, we get that
\begin{equation*}
F(n) \geq \frac{a(n^t - n)}{b^t - b} = \frac{a}{b^t - b}n^t + O(n^{t-1}),
\end{equation*}
as desired.
\end{proof}

Combining (\ref{eq-random}) with (\ref{eq-solved-recurrence}), it suffices
to find a coloring of some $K_b$ with $a$ rainbow copies of $H$, where
\begin{equation}\label{eq-y-criterion}
a > \frac{b(b^{t-1} - 1){r\choose{e}}e!}{|\text{Aut}(H)|r^e}.
\end{equation}

Before stating and proving the main result, we give an
explicit example that serves as a foundation and motivation 
for some of the 
techniques used in the proof.
Note that the following was proved for $r = 3$ by Erd\H{o}s and Hajnal in 1972 \cite{Erdos-Hajnal}.

\begin{proposition}\label{prop-K3}
$K_3$ is $r$-rainbow-uncommon for all $r \geq 3$.
\end{proposition}

\begin{proof}
We will
define a coloring of $E(K_r)$ in which every triangle is rainbow
and then use the iterated blowup method to determine a lower bound of the maximum
number of rainbow triangles the $K_n$ can contain.
For such coloring of $E(K_r)$, each color class must form a matching.
We call it a \textbf{parallel $r$-coloring}.
Examples with $r = 7$ and $r = 8$ are provided 
in Figure \ref{fig-parallel-coloring}.
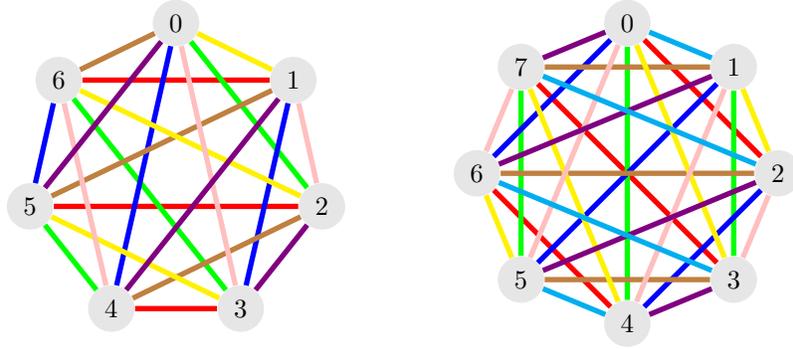
\begin{figure}
    \centering
    \begin{tikzpicture}[scale=2]
    \node[circle,fill=gray!20] at (0,1) (a) {0};
    \node[circle,fill=gray!20] at (0.78,0.62) (b) {1};
    \node[circle,fill=gray!20] at (0.97,-0.22) (c) {2};
    \node[circle,fill=gray!20] at (0.43,-0.9) (d) {3};
    \node[circle,fill=gray!20] at (-0.43,-0.9) (e) {4};
    \node[circle,fill=gray!20] at (-0.97,-0.22) (f) {5};
    \node[circle,fill=gray!20] at (-0.78,0.62) (g) {6};

    \draw[red,line width = 2pt] (g)--(b) (c)--(f) (d)--(e);
    \draw[green,line width = 2pt] (a)--(c) (d)--(g) (e)--(f);
    \draw[blue,line width = 2pt] (b)--(d) (e)--(a) (f)--(g);
    \draw[brown,line width = 2pt] (c)--(e) (f)--(b) (g)--(a);
    \draw[yellow,line width = 2pt] (d)--(f) (g)--(c) (a)--(b);
    \draw[pink,line width = 2pt] (e)--(g) (a)--(d) (b)--(c);
    \draw[violet,line width = 2pt] (f)--(a) (b)--(e) (c)--(d);

    \foreach \t in {0,1,...,7}{
        \node[circle,fill=gray!20] at ($(3,0) + (90 - 45*\t:1)$) (\t) {\t};
    }
    \draw[red, line width = 2pt] (2)--(0) (3)--(7) (4)--(6);
    \draw[green, line width = 2pt] (3) -- (1) (4) -- (0) (5) -- (7);
    \draw[blue, line width = 2pt] (4)--(2) (5)--(1) (6)--(0);
    \draw[brown, line width = 2pt] (5)--(3) (6)--(2) (7)--(1);

    \draw[yellow, line width = 2pt] (6)--(5) (7)--(4) (0)--(3) (1) -- (2);
    \draw[pink, line width = 2pt] (7) -- (6) (0) -- (5) (1) -- (4) (2) -- (3);
    \draw[violet, line width = 2pt] (0)--(7) (1)--(6) (2)--(5) (3) -- (4);
    \draw[cyan, line width = 2pt] (1)--(0) (2)--(7) (3)--(6) (4) -- (5);
    \end{tikzpicture}
    \caption{Parallel $r$-coloring defined in the proof of Proposition \ref{prop-K3}.
    In particular, there can be no non-rainbow triangles.}
    \label{fig-parallel-coloring}
\end{figure}

Such coloring can be constructed for all $r\geq 3$, and it is convenient to
describe the parallel $r$-coloring for $r$ odd
and even separately. Let $c$ denote the $r$-coloring of $E(K_r)$ 
with vertices $\{0, 1, \ldots, r-1\}$. 
\begin{itemize}
    \item when $r$ is odd, for each $k=1,\ldots, 
    r$ color $c(\{k+i \mod{r}, k-i \mod{r}\})=k$ for all $i=1,\ldots, \frac{r-1}{2}$,
    \item when $r$ is even, for each $k=1,\ldots, \frac{r}{2}$ 
    color $c(\{k+i \mod{r}, k-i \mod{r}\})=k$ for all $i=1,\ldots, \frac{r}{2}-1$ and for each $k=\frac{r}{2}+1,\ldots, r$ color $c(\{k+i \mod{r}, k-i+1 \mod{r}\})=k$ for all $i=1,\ldots, \frac{r}{2}$.
\end{itemize} 
It follows that no vertex has two incident edges of the same color, and hence every
$K_3$ must be rainbow.
By the proof of Lemma~\ref{lem-main}, the number of rainbow copies of $K_3$ in a $K_n$ generated by the iterated
blow-up of $K_r$ with coloring $c$ is at least $\frac{r-2}{6(r+1)}(n^3 - n)$, while the expected number of rainbow copies in a uniformly random coloring is $\frac{(r-1)(r-2)}{6r^2}n(n-1)(n-2)$. It is straightforward
to check that this construction satisfies the inequality (\ref{eq-y-criterion}).  
\end{proof}

The next two results give us an upper bound the number of non-rainbow $K_t$s in $K_r$ under a parallel $r$-coloring.

\begin{lemma}\label{num_K4s} 
Let $r\geq 6$, and let $c:E(K_r)\rightarrow \{1,\dots, r\}$ be a parallel $r$-coloring. Then there are at most~$r\binom{\lfloor r/2\rfloor}{2}$ non-rainbow copies of $K_4$ in $K_{r}$ under $c$. \end{lemma}

\begin{proof}
    Given $r\geq 6$, let $c$ be a parallel $r$-coloring of $E(K_r)$. Notice that under this coloring, every non-rainbow $K_4$ in $K_r$ uses exactly two edges from some color class. Therefore, for each color, there are at most $\binom{\lfloor r/2\rfloor}{2}$ non-rainbow copies of $K_4$ that have that color repeated. Thus, since there are $r$ colors in total, the number of non-rainbow $K_4$s in $K_r$ is at most $r\binom{\lfloor r/2\rfloor}{2}$. 
\end{proof}

\begin{lemma}\label{num_Kts} Let $t\geq 4$, and let $r\geq {{t}\choose{2}}$ and $c:E(K_r)\rightarrow \{1,\dots,r\}$ be a parallel $r$-coloring. Then there are at most
$r\binom{\lfloor r/2\rfloor}{2}\binom{r-4}{t-4}$ non-rainbow copies of $K_{t}$ in $K_{r}$ under $c$. \end{lemma}

\begin{proof}
Given $t\geq 4$ and $r\geq {{t}\choose{2}}$, let $c$ be a parallel $r$-coloring of $E(K_r)$. Notice that every non-rainbow $K_t$ in $K_r$ contains a non-rainbow $K_4$. By Lemma~\ref{num_K4s}, there are at most $r\binom{\lfloor r/2\rfloor}{2}$ non-rainbow $K_4$s in $K_r$ under the parallel $r$-coloring, and each non-rainbow $K_4$ is contained in at most $\binom{r-4}{t-4}$ $K_t$s because we may select the remaining $t-4$ vertices of $K_t$ from $K_r$ in $\binom{r-4}{t-4}$ ways. Therefore, the number of non-rainbow copies of $K_t$ in $K_r$ is at most $r\binom{\lfloor r/2\rfloor}{2}\binom{r-4}{t-4}$ under $c$.
\end{proof}

We are now ready to prove our main result, restated (and adjusted slightly) below.

\begin{theorem}
Let $t\geq 4$. There exists $r_t \geq {t\choose{2}}$ such that for all $r\geq r_t$, $K_t$ is $r$-rainbow-uncommon.     
\end{theorem}

\begin{proof}
Let $t\geq 4$ and $r\geq {{t}\choose{2}}$, and let $c$ be a parallel $r$-coloring of $E(K_r)$. By Lemma~\ref{num_Kts}, there are at most $r\binom{\lfloor r/2\rfloor}{2}\binom{r-4}{t-4}$ non-rainbow copies of $K_t$ in $K_r$ under $c$. Therefore, there are at least $$\binom{r}{t}-r\binom{\lfloor r/2\rfloor}{2}\binom{r-4}{t-4}\geq 
\binom{r}{t}\left(1-\frac{r \cdot t!}{8(t-4)!(r-1)(r-3)}\right)$$ rainbow copies of $K_t$ in $K_r$. In the inequality above, note that $\binom{\lfloor r/2\rfloor}{2}\leq \frac{r(r-2)}{8}$.

We will use the iterated blow-up method described at the beginning of this section with parameters $H=K_t$, $b=r$, and $a=\binom{r}{t}\left(1-\frac{r \cdot t!}{8(t-4)!(r-1)(r-3)}\right)$. Note that this implies $|Aut(H)|=t!$ and $e=\binom{t}{2}$. Therefore, it suffices to show that the inequality~(\ref{eq-y-criterion}) is satisfied, i.e. we will show that \begin{equation}\label{ineq_complicated}
    \binom{r}{t}\left(1-\frac{r \cdot t!}{8(t-4)!(r-1)(r-3)}\right) > \frac{r(r^{t-1} - 1){r\choose{\binom{t}{2}}}\binom{t}{2}!}{t!r^{\binom{t}{2}}}.\end{equation}

Rearranging the terms in the inequality~(\ref{ineq_complicated}), one can show that it is equivalent to 

$$\left(1-\frac{r\cdot t!}{8(t-4)!(r-1)(r-3)}\right)r^{t(t-1)/2}>\frac{(r-t)!}{(r-\binom{t}{2})!}(r^t-r).$$

Notice that the term $\frac{(r-t)!}{(r-\binom{t}{2})!}$ on the right hand side of the inequality above can be written as $$(r-t)(r-t-1)\cdots\left(r-\binom{t}{2}+1\right) =\prod_{l=t}^{\binom{t}{2}-1}\left(r-l\right), $$ therefore, the inequality~(\ref{ineq_complicated}) holds if and only if the following inequality holds:

\begin{equation}\label{ineq_semifinal}    
\left((r-1)(r-3)-\frac{r\cdot t!}{8(t-4)!}\right)r^{t(t-1)/2}-(r-1)(r-3) \left(\prod_{l=t}^{\binom{t}{2}-1}\left(r-l\right)\right)(r^t-r)>0.\end{equation}

Consider the leading coefficient of the polynomial on the left hand side of (\ref{ineq_semifinal}). Notice that the term with the largest power of $r$ that appears in (\ref{ineq_semifinal}) is $r^{2 + t(t-1)/2}$, however its coefficient is $0$. Therefore, we will focus on the coefficient of the term $r^{1 + t(t-1)/2}$, which is 

\begin{equation}\label{ineq_final}
\left(-4-\frac{t!}{8(t-4)!}\right) - \left(-4 - \left(\sum_{l = t}^{\binom{t}{2}-1}l\right)\right) = t(t-1)(t-3)/2.\end{equation}

Since $t\geq 4$, the leading coefficient calculated in (\ref{ineq_final}) is strictly positive. This implies that there exists a large enough integer $r_t$ that satisfies the inequality (\ref{ineq_semifinal}), in fact, the inequality (\ref{ineq_semifinal}) will hold for all $r\geq r_t$. Therefore, $K_t$ is $r$-rainbow uncommon for all $r\geq r_t$.
\end{proof}

Theorem \ref{thm-weak} follows directly from the theorem above. Theorem \ref{thm-exact} follows from Proposition \ref{prop-K3} and by substituting $t = 4, r= 6$ and $t = 5, r = 10$, into 
(\ref{ineq_semifinal}) and seeing that both cases satisfy the inequality.

\section{Conclusion}

Sun very recently showed that any graph containing a cycle is $r$-rainbow-uncommon for all $r \geq e$ \cite{Sun}.
While that implies our result, we explicitly construct
the colorings for complete graphs that show rainbow-uncommonness.

Future work is to investigate upper bounds, particularly to show which graphs are $r$-rainbow-common. Thus far, the only class of graphs known
to be rainbow-common are disjoint unions of stars \cite{DeSilva}.
It is also likely that $P_4$, a path on four vertices, is $3$-rainbow-common.
We believe the converse of Sun's result is also true,
formulated in the following conjecture.
\begin{conjecture}
$H$ is rainbow-common if and only if $H$ is a forest.
\end{conjecture}
By \textit{rainbow-common}, we mean $r$-rainbow-common for all 
$r \geq e$.

\section*{Acknowledgements}

This work was done as part of the 2021 Polymath Jr program, partially supported by NSF award DMS–2113535. We are grateful to the organizers of the Polymath Jr. program for creating the environment to carry out this research project. We also thank our colleagues Pablo Blanco and Sarvagya Jain who contributed to helpful conversations and questions that were valuable for the present paper. The work of the second author was partially supported by the Fairfield Summer Research award.

\bibliographystyle{amsplain}
\bibliography{references}

\providecommand{\bysame}{\leavevmode\hbox to3em{\hrulefill}\thinspace}
\providecommand{\MR}{\relax\ifhmode\unskip\space\fi MR }
\providecommand{\MRhref}[2]{%
  \href{http://www.ams.org/mathscinet-getitem?mr=#1}{#2}
}
\providecommand{\href}[2]{#2}
\begin{thebibliography}{1}

\bibitem{Conlon}
David {Conlon}, Jacob {Fox}, and Benny {Sudakov}, \emph{{Recent developments in
  graph Ramsey theory}}, arXiv e-prints (2015), arXiv:1501.02474.

\bibitem{Erdos-Hajnal}
Paul Erd\H{o}s and András Hajnal, \emph{On ramsey like theorems, problems and
  results}, Combinatorica (1972).

\bibitem{Goodman}
A.~W. Goodman, \emph{On sets of acquaintances and strangers at any party}, The
  American Mathematical Monthly \textbf{66} (1959), no.~9, 778--783.

\bibitem{DeSilva}
Jessica~De Silva, Xiang Si, Michael Tait, Yunus Tunçbilek, Ruifan Yang, and
  Michael Young, \emph{Anti-ramsey multiplicities}, Australas. J Comb.
  \textbf{73} (2019), 357--371.

\bibitem{Sun}
Yihang Sun, \emph{Rainbow common graphs must be forests},  (2023).

\end{thebibliography}

\end{document}